\documentclass[12pt]{article}
\usepackage{amsmath,amsthm,amssymb,latexsym,color}
\usepackage{hyperref}
\usepackage{bbm}
\theoremstyle{plain}
\newtheorem{theor}{Theorem}[section]
\newtheorem{claim}[theor]{Claim}

\newtheorem{definition}[theor]{Definition}

\newtheorem{lemma}[theor]{Lemma}
\newtheorem{rem}[theor]{Remark}
\theoremstyle{remark}

\def\theorparam{\beta_n}
\def\M{{M_n}}
\def\F{{F_{ij}}}
\def\Fp{{F^{\perp}_{ij}}}
\def\Aset{{\mathcal{M}_{n,d}}}
\def\An{A_n}
\def\A{A}
\def\a{{a}}
\def\R{{\mathbb R}}

\def\Prob{{\mathbb P}}

\def\col{{\rm Col}}
\def\row{R}

\def\supp{{\rm supp }}

\def\spn{{\rm span}\,}

\def\Om0{\Omega_0}

\newcommand{\rk}{{\rm rk\,}}

\newcommand{\lam}{\lambda}
\newcommand{\eps}{\varepsilon}

\def\r{\right}

\newcommand{\en}{\mathcal{E}_{n-1}}
\newcommand{\enn}{\mathcal{E}_{n-2}}
\newcommand{\e}{\mathcal{E}}

\def\O1{\Omega_1(\eps_1)}

\date{}

\title{The rank of random regular digraphs of constant degree}
\author{
Alexander E. Litvak
\and
Anna Lytova
\and
Konstantin Tikhomirov
\and
Nicole Tomczak-Jaegermann
\and
Pierre Youssef
}
\newcommand\address{\noindent\leavevmode

\medskip
\noindent
Alexander E. Litvak
and Nicole Tomczak-Jaegermann,\\
Dept.~of Math.~and Stat.~Sciences,\\
University of Alberta, \\
Edmonton, AB, Canada, T6G 2G1.\\
\texttt{\small
e-mails:  aelitvak@gmail.com \, \, and \, \,
nicole.tomczak@ualberta.ca}\\

\medskip

\noindent
Anna Lytova,\\
Faculty of Math., Physics, and Comp. Science,\\
University of Opole,\\
ul. Oleska 48, 45-052,\\
Opole, Poland.\\
\texttt{\small
e-mail:
alytova@math.uni.opole.pl
}\\

\medskip

\noindent
 Konstantin Tikhomirov,\\
Dept.~of Math.,
Princeton University,\\
Fine Hall, Washington road,\\
Princeton, NJ 08544.\\
\texttt{\small
e-mail:   kt12@math.princeton.edu}\\

\medskip

\noindent
Pierre Youssef,\\
Universit\'e Paris Diderot,\\
Laboratoire de Probabilit\'es et de mod\`eles al\'eatoires,\\
75013 Paris, France.\\
\texttt{\small
e-mail:  youssef@math.univ-paris-diderot.fr}
}

\begin{document}

\maketitle

\abstract{
Let $d$ be a (large) integer.
Given $n\geq 2d$, let $\An$ be the adjacency matrix
of a random directed $d$-regular graph on $n$ vertices, with
the uniform distribution. We show that the rank of $\An$ is at least $n-1$
with probability going to one as $n$ grows to infinity.
The proof combines the well known method of simple switchings and a
recent result of the authors on delocalization of eigenvectors of $\An$.
}

\bigskip

\noindent
{\small \bf AMS 2010 Classification:}
{\small
primary: 60B20, 15B52, 46B06, 05C80;\\
secondary: 46B09, 60C05
}


\section{Introduction}

Singularity of random discrete square matrices is a subject with a long
history and many results and applications. In particular, quantitative estimates on
the smallest singular number are important for understanding complexity of some algorithms.
Well invertible sparse matrices are of general interest in computer science, and it is known
that sparse matrices are computationally more efficient (require less operations for matrix-vector
multiplication). In this paper we deal with sparse random square matrices from a certain model.

In a standard setting, when the entries of the $n\times n$ matrix are
i.i.d. Bernoulli $\pm 1$ random variables,
the invertibility problem has been addressed by Koml\'os in \cite{Kom1,Kom2},
and later considered in several papers \cite{KKS,TV-sing,BVW}.
A long-standing conjecture asserts that the probability that the Bernoulli  matrix is singular
is $\big(1/2+o(1)\big)^{n}$. Currently, the best upper bound on this probability
 is $\big(1/\sqrt{2} + o(1)\big)^n$ obtained by  Bourgain, Vu, and Wood \cite{BVW}.
We would also like to mention related works on singularity of symmetric Bernoulli
matrices \cite{CTV,Ngu1,Ver} and Nguyen's work \cite{Ngu2}, where
random $0/1$ matrices with independent rows and row-sums constraints were considered.

A corresponding question can be formulated for adjacency matrices of random graphs.
For instance, consider the adjacency matrix of an undirected Erd\H{o}s--Renyi random graph
$G(n,p)$ which is a symmetric random $n\times n$ matrix whose off-diagonal
entries are i.i.d. $0/1$ random variables with the parameter $p$.
The case $p=1/2$ is closely related to the random model from the previous paragraph.
In \cite{CV} Costello and Vu proved that, given $c>1$, with large probability the rank
of the adjacency matrix of $G(n,p)$ is equal to the number of non-isolated vertices
whenever $c\ln n/n \leq p\leq 1/2$. It is known that $p=\ln n/n$ is the threshold of
connectivity, so that when $c>1$ and $c\ln n/n \leq p\leq 1/2$, the graph $G(n,p)$
typically contains no isolated vertices and is therefore of full rank with probability
going to one as $n$ tends to infinity (see \cite{BR} for quantitative bounds in the non-symmetric 
setting).
It was also shown that
if $p\to 0$ and $np\to \infty$, then $\big(\rk G(n,p)\big)/n \to 1$ as $n$ goes to infinity,
where $\rk (A)$ stands for the rank of the matrix $A$.
The case $p=y/n$ for a fixed $y$ was studied in \cite{BLS} where asymptotics for
$\big(\rk G(n,p)\big)/n$ were established.

In the absence of independence between the matrix entries, the problem of singularity
involves additional difficulties. Such a problem was considered for
the (symmetric) adjacency matrix $\M$ of a random (with respect to the uniform probability)
undirected $d$-regular graph on $n$ vertices, i.e., a graph in which each vertex has precisely $d$ neighbours.
The case $d=1$ corresponds to a permutation matrix which is non-singular,
and for $d=2$ the graph is a union of cycles and the matrix is almost surely singular.
Moreover, the invertibility of the adjacency matrix of the complementary graph is equivalent to
that of the original one
(in fact, the ranks of the adjacency matrices of a $d$-regular graph and of its complementary graph are the same).
This can be seen by first noticing that the eigenvalues of $J_n-\M$, where $J_n$ is the $n\times n$ matrix of ones,
are equal to the difference between those of $J_n$ and those of $\M$ (since the two commute)
and that all eigenvalues of $\M$ are bounded in absolute value by $d$,
which is smaller than the only non-zero eigenvalue of $J_n$ (equals to $n$).
In parallel to the Erd\H{o}s--Renyi model,  Costello and Vu  raised the following problem:
``For what $d$ is the adjacency  matrix $\M$ of full rank almost surely?" (see \cite[Section~10]{CV}).
They conjectured that for every $3\leq d\leq n-3$,
the adjacency matrix $\M$ is non-singular with probability going to $1$ as $n$ tends to $\infty$.
This conjecture was mentioned again in the survey
\cite[Problem~8.4]{Vu-survey} and 2014 ICM talks by Frieze \cite[Problem~7]{Frieze} and
by Vu \cite[Conjecture~5.8]{Vu-ICM}.

In the present paper, we are interested in behaviour of adjacency matrices of
random directed $d$-regular graphs with the uniform model, that is, random graphs
uniformly distributed on the set of all directed $d$-regular graphs on $n$ vertices.
By a directed $d$-regular graph  on $n$ vertices we mean a graph such that
each vertex has precisely $d$ in-neighbours and $d$ out-neighbours and where
loops and $2$-cycles are allowed but multiple edges are prohibited.
The adjacency matrix $\An$ of such a graph
is uniformly distributed on the set of all (not necessarily symmetric)
$0/1$ matrices  with $d$ ones in every row and every column.
As in the symmetric case, in the case $d=1$ the matrix $A_1$ is a permutation matrix which is non-singular,
and in the case  $d=2$ the matrix $A_2$ is almost surely singular.
It is natural to ask the same question as in \cite{CV} for directed $d$-regular graphs
(see, in particular, \cite[Conjecture~1.5]{Cook}).
Cook \cite{Cook} proved that such a matrix is asymptotically almost surely non-singular
for $\omega(\ln^2 n)\leq d\leq n-\omega(\ln^2 n)$, where $f=f(n)= \omega(a_n)$
means $f/a_n \to \infty$ as $n\to \infty$. Further, in \cite{LLTTY1,LLTTY2},
the authors of the present paper showed that the singularity probability is bounded above
by $C\ln^3 d/\sqrt{d}$ for  $C\leq d\leq  n/\ln^2 n$, where $C$ is a (large) absolute positive
constant.
This settles the problem of singularity for  $d=d(n)$ growing to infinity with $n$ at any rate.
Moreover, quantitative bounds on the smallest singular value
for this model were derived in \cite{Cook2} and \cite{LLTTY3}.
Those estimates turn out to be essential in the study of
the limiting spectral distribution \cite{Cook2, LLTTY5}.

The challenging case when $d$ is a constant remains unresolved
and is  the main motivation for writing this note.
The lack of results in this setting
constitutes a major obstacle in establishing the conjectured non-symmetric (oriented) Kesten--McKay law as the limit
of the spectral distribution for the directed random $d$-regular graph (see, in particular, \cite[Section~7]{BC}).
This note illustrates a  partial progress in this direction.
Our main result is the following theorem.
Note that the probability bound in it is non-trivial only if
$\ln n> C\ln^2 d$, however in the complementary case we have  $\rk (\An)= n$
with high probability as was mentioned above.

\begin{theor}\label{th: main}
There exists a universal constant $C>0$ such that for any integer $d\geq C$ the following holds.
Let $n>d$ and let $\An$ be the adjacency matrix of the random directed $d$-regular graph on $n$ vertices,
with uniform distribution allowing loops but no multiple edges.
Then
$$
  \Prob\{\rk \An \geq n-1\}\geq 1-C\ln^2 d/\ln n.
$$
\end{theor}

This theorem is ``one step away'' from proving the conjectured invertibility for a (large) constant $d$.
We would like to emphasize that the main point of the theorem is that  even for a constant $d$ the
probability of a ``good" event tends to $1$ with $n$ (and not with $d$ as in \cite{LLTTY1, LLTTY2}).
To the best of our knowledge, it is the first result of such a kind dealing with singularity of
$d$-regular random matrices.
The proof of Theorem~\ref{th: main} uses  the standard technique of
simple switchings (in particular, it was also used in \cite{Cook} and \cite{LLTTY2}).
We recall the procedure using the matrix language.
Denote by $\Aset$ the set of all adjacency matrices of directed $d$-regular graphs on $n$ vertices, i.e.,
all $0/1$ matrices with $d$ ones in every row and every column.
Given $\A=(\a_{st})_{1\leq s,t\leq n} \in \Aset$ we
say that a switching in $(i, j, k, \ell)$  {\it can be performed}
if $\a _{ik}=\a _{j\ell}=1$ and $\a_{i\ell}=\a _{jk}=0$. Further, given such a matrix $\A\in \Aset$, we
say that a matrix $\bar A = (\bar \a _{st})_{s,t} \in \Aset$ is obtained
from $A$ by a simple switching (in  $(i, j, k, \ell)$) if
$\bar \a _{ik}=\bar \a _{j\ell}=0$,  $\bar \a _{i\ell}=\bar \a _{jk}=1$,
and $\bar \a _{st}=\a_{st}$ otherwise. Note that this operation
does not destroy the $d$-regularity of the underlying graph.
A well known application of the simple switching is due to McKay
\cite{Mckay} in the context of undirected $d$-regular graphs.
Starting from a matrix $\A\in \Aset$, one can reach any other matrix in $\Aset$
by iteratively applying simple switchings.
In this connection, a feasible strategy in estimating the cardinality of a subset $\mathcal{B}\subset \Aset$
is to pick an element in $\mathcal{B}$ and bound the number of switchings
which would result in another element of $\mathcal{B}$ versus switchings leading outside of $\mathcal{B}$.
In a sense, one  studies the stability of $\mathcal{B}$ under this operation.
We will make this standard approach more precise in the preliminaries.
Clearly, for a matrix $\A\in \Aset$ and any $1\leq i<j\leq n$,
$$
\rk \A= \dim\,\big( \spn\{(\row_s)_{s\neq i,j}, \row_i+\row_j, \row_i\}\big),
$$
where $\row_1, \row_2, ..., \row_n$ denote the rows of $\A$. Since $(\row_s)_{s\neq i,j}$
and $\row_i+\row_j$ are invariant under any switching involving the $i$-th and $j$-th rows,
then for any matrix $\bar \A$ obtained from $\A$ by such a switching, we have
\begin{equation}\label{eq: rank}
\vert \rk \A - \rk \bar\A \vert \leq 1.
\end{equation}
In a sense, we will show that given a matrix from $\Aset$ of corank at least $2$, most of simple switchings
tend to increase the rank.
We will use that the kernel of A, $\ker \A$, is contained in $\Fp$, where
$\F:= \spn\{(\row_\ell)_{\ell\neq i,j}, \row_i+\row_j\}$. Note that $\F$ is invariant under any simple
switching on the $i$-th and $j$-th rows.

In this paper, the  simple switching procedure is combined
with a recent delocalization result for eigenvectors of $\An$ established by the authors
in \cite[Corollary~1.2]{LLTTY4}.
Below we state a less general version of the delocalization result.

\begin{theor}[\cite{LLTTY4}]\label{th: deloc}
There exists a universal constant $C>0$ such that for any integer $d\geq C$ the following holds.
Let $n> d$ and let $\An$ be the adjacency matrix
of the directed random $d$-regular graph on $n$ vertices. Then with
probability at least $1-2/n$ any vector
$x\in \left(\ker \An\, \cup \, \ker \An^T\r) \setminus\{0\}$
satisfies
$$
 \forall \lambda\in \R \quad \quad \vert\{i\leq n:\, x_i=\lambda\}\vert \leq Cn\, \ln^2 d/\ln n.
$$
\end{theor}

In fact in \cite{LLTTY4} the assumption $d\leq \exp(c\sqrt{\ln n})$ was also involved,
however for  $d\geq \exp(c\sqrt{\ln n})$  the bound on the cardinality trivially holds.
A more general quantitative version of Theorem~\ref{th: deloc}, proved in \cite{LLTTY4},
served as the key element in establishing the circular law \cite{LLTTY5} for the limiting spectral distribution
when the degree $d=d(n)\leq \ln ^{96} n$ tends to infinity with $n$ (for the regime $d>\ln ^{96}n$ see
\cite{Cook2}).
 In this note, we take advantage of the fact that the results of \cite{LLTTY4} also work for any large constant $d$.

\section{Preliminaries}

For an $n\times n$ matrix $\A$, we denote by $(\row_s)_{s\leq n}$ and $(\col_s)_{s\leq n}$ its rows and columns respectively.
Given positive integer $m$, we denote by $[m]$ the set $\{1, 2, ..., m\}$.
Further, for a vector $x\in \R^n$, we denote its support by $\supp\, x=\{i\leq n:\, x_i\neq 0\}$.

Given two sets $\mathcal{B}$, $\mathcal{B}'$ and a relation
$Q\subset \mathcal{B}\times \mathcal{B}'$, we set $Q(\mathcal{B})=\bigcup_{b\in \mathcal{B}} Q(b)$
and  $Q^{-1}(\mathcal{B}')=\bigcup_{b'\in \mathcal{B}'} Q^{-1}(b')$, where
$$
Q(b)= \{b'\in \mathcal{B}':\, (b,b')\in Q\}\quad \text{ and }\quad   Q^{-1}(b')= \{b\in \mathcal{B}:\, (b,b')\in Q\},
$$
for any $b\in \mathcal{B}$ and any $b'\in \mathcal{B}'$.
In what follows, we consider the symmetric relation $Q_0$ on $\Aset\times \Aset$
defined by
\begin{equation}\label{Q definition}
\mbox{$(\A, \bar \A)\in  Q_0\,\, $ if and only if $\,\, \bar \A$ can be obtained from $\A$ by a simple
switching.}
\end{equation}

The following simple claim will be used to compare cardinalities of two sets given a relation on their Cartesian product.
We refer to \cite[Claim~2.1]{LLTTY2} for a proof of a similar claim.
\begin{claim}\label{multi-al-gen}
Let $Q$ be a finite relation on $\mathcal{B}\times\mathcal{B}'$ such that for every $b\in \mathcal{B}$
and every $b'\in \mathcal{B}'$ one has $\vert Q(b)\vert\geq s_b$ and $\vert Q^{-1}(b')\vert \leq t_{b'}$
for some numbers $s_b,t_{b'}\geq 0$. Then
$$
\sum_{b\in\mathcal B}s_b \leq \sum_{b'\in \mathcal B'}t_{b'}.
$$
\end{claim}

Next, given $\A\in \Aset$, we estimate the number of
possible switchings on $\A$, that is, the cardinality of the set
$$
{\cal{F}}_\A= \{(i, j, k, \ell) \, :\,   \mbox{a switching in }\,
(i, j, k, \ell) \, \mbox{ can be performed} \}.
$$
Recall that we say that a simple switching {\it can be performed in}
$(i, j, k, \ell)$ if $\a _{ik}=\a _{j\ell}=1$ and $\a_{i\ell}=\a _{jk}=0$.
Note that this automatically implies that $i\neq j$ and $k\neq \ell$.
Note also that two formally distinct simple switchings $(i,j,k,\ell)$
and $(j,i,\ell,k)$ result in the same transformation of a matrix.

\begin{lemma}\label{l: nb-switch} Let $1\leq d\leq n$ and $\A\in \Aset$. Then
$$
n(n-d)d^2 - nd(d-1)^2 \leq | {\cal{F}}_\A| \leq  n(n-d)d^2.
$$
\end{lemma}

\begin{proof}
To find a possible switching, we first fix an entry $a_{ik}$ equal to $1$.
By $d$-regularity of $A$,  there are exactly $nd$ choices of the pair $(i,k)$.
To be able to perform a simple switching in $(i, j, k, \ell)$, the pair of indices $(j,\ell)$ must satisfy
$$
 a_{j\ell}=1 \quad \quad \mbox{ and } \quad \quad
 (j,\ell)\notin T:=\Big([n]\times\supp\, \row_i\Big)\bigcup \Big(\supp\, \col_k\times [n]\Big).
$$
By $d$-regularity we observe that the number
$p$ of pairs  $(s,t)\in T$ with $a_{st}=1$
satisfies $d^2\leq p\leq d^2+(d-1)^2$. Since there are $nd$ choices for $(j,\ell)$ with $a_{j\ell}=1$,
we observe that the number $q$   of pairs  $(s,t)\notin T$ with $a_{st}=1$ satisfies
$$(n-d)d -(d-1)^2 \leq q \leq (n-d)d. $$
Since $| {\cal{F}}_\A| = nd q$, we obtain the desired result.
\end{proof}

\begin{rem}
Note that for $d=1$, i.e., in the case of a permutation matrix,
the upper and lower bounds in the above lemma coincide and both equal $n(n-1)$.
More generally, assume $n=md$ for an integer $m$ and consider the block-diagonal matrix
$\A$ with $m$ $d\times d$ blocks, each block consisting of ones. Then a switching
in $(i, j, k, \ell)$ can be performed if and only if $i$, $j$ correspond to different
blocks (there are $m(m-1)d^2=n(n-d)$ such pairs) and $k\in \supp R_i$, $\ell \in \supp R_j$
(there are $d^2$ such choices). Thus for such  a matrix $\A$ we have
$$
  | {\cal{F}}_\A| = n(n-d)d^2,
$$
which corresponds to the upper bound in Lemma~\ref{l: nb-switch}.
\end{rem}

Denote by $\e_{\ref{th: deloc}}$ the event in Theorem~\ref{th: deloc}.
As usual, we don't distinguish between events for the uniformly distributed random matrix on $\Aset$
and corresponding subsets of $\Aset$. In particular, denoting $\beta_n := Cn\, \ln^2 d/\ln n$,
$$
\e_{\ref{th: deloc}}:=\{\A\in \Aset\,:\, \forall x\in\left(\ker \A \, \cup \, \ker \A^T\r) \setminus\{0\}
\, \, \,
\forall\lambda\in\R \quad
|\{i\le n :\,x_i=\lambda\}|\leq  \theorparam\bigr\},
$$
where $\theorparam:=\min(n,  Cn\, \ln^2 d/\ln n)$ and $C$ is the constant from
Theorem~\ref{th: deloc}. Further, for every  $r\leq n$ set
$$
   \e_r=\{\A\in\Aset\,:\, \rk \A\leq r\} \quad \quad \mbox{ and } \quad \quad
   E_r=\{\A\in\Aset:\,\rk \A= r\}.
$$

Given $\A\in \Aset$ and  $i\neq j$, we set
$$
 \F=\F(\A):=\spn\{(\row_s)_{s\neq i,j}, \row_i+\row_j\}.
$$
Clearly, $\ker(\A)\subset\Fp$.
We will be interested in those pairs $(i,j)$ for which this inclusion turns to equality. Given $\A\in \Aset$, define
$$
 K_\A= \big\{(i,j)\in [n]^2 : \, i\neq j \text{ and } \ker \A = \Fp(\A)\big\}.
$$

\begin{lemma}\label{ichoice}
Let $d<n$ and  $\A\in \en \cap \e_{\ref{th: deloc}}$.
Assume $\beta_n := Cn\, \ln^2 d/\ln n \leq n$.  Then $|K_\A |\geq (n - \theorparam)^2$.
\end{lemma}

\begin{proof}
Since $\A$ is singular there exists $y\in\R^n\setminus\{0\}$ such that
$$
  \sum _{s=1}^n y_s \row_s = 0.
$$
Since $y\in \ker A^T$ and $\A\in \e_{\ref{th: deloc}}$, the set $I:=\{i :\, y_i\ne 0\}$ is of
cardinality at least $n-\theorparam$. Note that if $i\in I$ then $\row_{i}\in \spn\{\row_s,\, s\neq i\}$,
therefore removing the $i$-th row keeps the rank unchanged, that is,
we have $\rk \A=\rk \A^i$, where $\A^{i}$ denotes the $n\times n$ matrix obtained by
substituting the $i$-th row of $\A$ with the zero row.

Fix $i\in I$. If $\A\in \enn$ then $\rk \A^i=\rk \A\leq n-2$.
Thus, the non-zero rows of $\A^i$ are linearly dependent, therefore there exists
$z\in \R^n\setminus\{0\}$ such that
$$
z_i=0\quad \text{ and }\quad  \sum _{s\ne i} z_s \row_s = 0.
$$
Clearly, $z\in \ker \A^T$ and by the condition $\A\in  \e_{\ref{th: deloc}}$, the set
$$
 J=J(i):=\{j\leq n :\, z_j\ne 0\}
$$
 is of
cardinality at least $n-\theorparam$.
Note that if $j\in J$, then $\row_j\in \spn\{\row_s,\, s\neq i,j\}$
and thus, $\row_i+\row_j\in \spn\{\row_s,\, s\neq i,j\}$.
This means that $(i,j)\in K_\A$. Thus for $\A\in \enn$ one has
$$
 \vert K_\A\vert \geq |I| \, \min_{i\in I} |J(i)| \geq (n-\theorparam)^2 .
$$

Now suppose that $\A\in E_{n-1}$ and fix  $i\in I$.
Since $\row_{i}\in \spn\{\row_s,\, s\neq i\}$,  there exist scalars $(x_s)_{s\neq i}$ such that
$$
\row_i= \sum_{s\neq i} x_s \row_s.
$$
Therefore setting $x_i=-1$, we have $x=(x_s)_{s\leq n}\in \ker \A^T$
and since $\A\in  \e_{\ref{th: deloc}}$,
the set $L=L(i):=\{j\leq n:\, x_j\neq -1\}$ is of cardinality at least $n-\theorparam$.
Note that if $j\in L$, then
$$
 \row_i+\row_j= (x_j+1)\row_j+ \sum_{s\neq i,j} x_s \row_s\not\in \spn\{\row_s,\, s\neq i,j\}
$$
(otherwise, we would have $\row_j\in \spn\{\row_s,\, s\neq i,j\}$,
which is impossible since $\rk \A =n-1$). Using again that $\rk \A =n-1$, we obtain that
$\dim \F=n-1$, that is,
$$
 \dim \Fp= 1=\dim \ker \A.
$$
Therefore the inclusion $\ker \A\subset \Fp$ implies that
$(i,j) \in K_\A$ and the lower bound on the cardinality of $K_\A$ follows.
\end{proof}

Note that for every $A\in \Aset$ the subspace $\Fp(\A)$
is invariant under simple switchings involving the $i$-th and $j$-th rows.
Moreover, for every pair $(i,j)\in K_\A$ one has $\ker \A=\Fp(\A)$.
Therefore, since our aim is to show that most switchings tend to increase the rank,
we need to eliminate those which keep this equality valid, that is,
those which  keep $\ker \A$ unchanged. This motivates the following definition.

\begin{definition}\label{propertyP}
Let $d<n$, $\A\in \Aset$, and $(i, j, k, \ell)\in {\cal{F}}_\A$.
Let $x\in \R^n$. We  say that a  switching in  $(i, j, k, \ell)$ is $x$-bad
if $x_k=x_{\ell}$.
In other words, a  switching in  $(i, j, k, \ell)$ is $x$-bad if $\A x=\bar \A x$
(where by $\bar \A$ we denote the new matrix obtained from $\A$ by the switching).
\end{definition}

In the next lemma we estimate the number of $x$-bad switchings.

\begin{lemma}\label{j-2choice}
Let $d<n$, $\beta_n = Cn\, \ln^2 d/\ln n$, $\A\in \en \cap \e_{\ref{th: deloc}}$ and $x\in \ker \A\setminus \{0\}$.
Then
$$\big|\big\{(i, j, k, \ell)\in {\cal{F}}_\A:\;\mbox{switching in $(i, j, k, \ell)$ is $x$-bad}\big\}\big|\leq n \theorparam d^2.$$
\end{lemma}

\begin{proof}
Let $\{\lam _p\, : \, p\leq m\}$ be the set of disctinct values taken by coordinates of $x$.
For every $p\leq m$ set
$$
  L_p = \{s\leq n \, :\, x_s = \lam _p\}.
$$
Since $\A\in \e_{\ref{th: deloc}}$, we have $|L_p|\leq \theorparam$
for all $p\leq m$. Since for an $x$-bad switching in $(i, j, k, \ell)$ we have
$x_k=x_{\ell}$,  $k$ and $\ell$ should belong to the same $L_p$. By $d$-regularity,
for every $p\leq m$ the number of switchings in $(i,j,k,\ell)$ with $k,\ell\in L_p$
is at most $d^2\vert L_p\vert^2$ (since we must have $a_{ik}=a_{j\ell}=1$).
Thus, the number of $x$-bad switchings is bounded above by
$$
  \sum _{p=1}^m \, d^2\vert L_p\vert^2 \leq   d^2\, \max_{p\leq m }\vert L_p\vert\,
  \sum _{p=1}^m \, \vert L_p\vert
  \leq  n\theorparam d^2.
$$
\end{proof}

\section{Proof of Theorem~\ref{th: main}}
\label{s: proof-main}

We start with the following lemma estimating the number of simple
switchings which increase the rank.

\begin{lemma}\label{n-2images}
Let $n\geq 2 d$ be large enough integers and $\A\in \en \cap \e_{\ref{th: deloc}}$.
Assume $\beta_n := Cn\, \ln^2 d/\ln n\leq n/4$.
Then there are at least $n(n-3\theorparam)d^2$
switchings in $(i,j,k,\ell)$ which increase the rank, i.e., for which
$$\rk \bar \A =\rk \A +1,$$
where $\bar \A$ denotes the matrix obtained by the switching.
\end{lemma}

\begin{proof}
Given  two rows $R_i$ and $R_j$, $i\neq j$, of $\A$, by $d$-regularity, there are at most
$d^2$ $4$-tuples $(i,j,k,\ell)$ in which a switching can be performed. Thus,
the number of switchings in $(i,j,k,\ell)$ with $(i,j)\in [n]\times [n]\setminus  K_\A$ is at most
$\vert K_\A^c\vert d^2$, where the complement is taken in $[n]^2$.
Therefore, applying Lemmas~\ref{l: nb-switch} and \ref{ichoice} we obtain that
the number $N$ of possible switchings in $(i,j,k,\ell)$
 with $(i,j)\in K_\A$ is at least
\begin{equation}\label{numbsw}
 N\geq | {\cal{F}}_\A| - \vert K_\A^c\vert d^2\geq
 n(n-2d)d^2 - \big(n^2 -(n-\theorparam)^2\big)d^2\geq
 d^2\big(n^2 - 2dn - 2\beta _n n +\beta_n^2\big).
\end{equation}

For the rest of the proof, we fix a non-zero vector $x\in \ker \A$.

Fix for a moment $(i,j)\in K_\A$,
and note that for any switching on the $i$-th and $j$-th rows, we have
$$
 \ker \bar \A \subset \Fp(\bar A) =\Fp(A) =\ker \A .
$$
Observe that if a switching on $i,j$-th rows is not $x$-bad,
then
$\bar \A x\neq \A x=0$ and therefore
$$
 \ker \bar \A \ne \ker\A = \Fp(\bar A) .
$$
This means that $\rk \bar\A > \rk \A $ and by \eqref{eq: rank} implies that $\rk \bar \A =\rk \A +1$.
Thus, any possible switching in $(i,j,k,\ell)$, which is not $x$-bad and such that $(i,j)\in K_\A$
increases the rank of the matrix by one.

Applying Lemma~\ref{j-2choice} and inequality \eqref{numbsw}, we obtain that
the number $N_0$ of switchings described above is at least
\begin{align*}
N_0 \geq N-  n \theorparam d^2\geq d^2\big(n^2 - 2dn - 3\beta _n n +\beta_n^2\big).
\end{align*}
Since $\theorparam^2\geq 2nd$ for
large enough $n$, this completes the proof.
\end{proof}

\begin{proof}[Proof of Theorem~\ref{th: main}]
We may assume that $d\leq \exp(c\sqrt{\ln n})$ for a small enough absolute constant $c>0$
(otherwise the probability bound in Theorem~\ref{th: main} trivially holds). In this
case $n\geq 4\theorparam$.
Fix $r\in\{1,\ldots, n-2\}$ and consider the relation
$$
 Q_r\subseteq (E_r\cap \e_{\ref{th: deloc}})\times E_{r+1},
$$
defined by $(\A, \bar \A)\in Q_r$ if and only if $\A\in E_r\cap \e_{\ref{th: deloc}}$, $\bar \A \in E_{r+1}$,
and $(\A, \bar \A)\in Q_0$, where the symmetric relation $Q_0$ is given by \eqref{Q definition}.

Using that any two switchings $(i,j,k,\ell)$ and $(j,i,\ell,k)$
produce the same transformed matrix, and applying Lemma~\ref{n-2images} we observe that for every
$\A\in E_r\cap \e_{\ref{th: deloc}}$,
$$
   |Q_r(\A)|\geq n(n-3\theorparam)d^2/2.
$$
Now let $\bar \A\in Q_r(E_r\cap \e_{\ref{th: deloc}})$. If $\bar \A\in \e_{\ref{th: deloc}}$, then by
Lemmas~\ref{n-2images} and \ref{l: nb-switch},
$$
|Q_r^{-1}(\bar\A)|\leq  \big(| {\cal{F}}_{\bar\A}|- n(n-3\theorparam)d^2\big)/2\leq 3n\theorparam d^2/2.
$$
Otherwise, if $\bar \A\in \e_{\ref{th: deloc}}^c$ then
$$
|Q_r^{-1}(\bar\A)|\leq | {\cal{F}}_{\bar\A}|/2\leq n(n-d)d^2/2.
$$
Then Claim~\ref{multi-al-gen} implies
$$
\frac{n(n-3\theorparam)d^2}{2}\, \vert E_r\cap \e_{\ref{th: deloc}}\vert\leq \frac{3n\theorparam d^2}{2}\,
\vert E_{r+1}\cap \e_{\ref{th: deloc}}\vert +
\frac{n(n-d)d^2}{2}\,  \vert E_{r+1}\cap \e_{\ref{th: deloc}}^c\vert.
$$
Summing over all $r=1,\ldots, n-2$ gives
$$
\vert \enn\cap \e_{\ref{th: deloc}}\vert\leq
\frac{3\theorparam}{n-3\theorparam}\,
\vert \en\cap \e_{\ref{th: deloc}}\vert +
\frac{n}{n-3\theorparam}\, \vert \en\cap \e_{\ref{th: deloc}}^c\vert
=
\frac{3\theorparam}{n-3\theorparam}\,
\vert \en \vert +
 \vert \e_{\ref{th: deloc}}^c\vert .
$$
Using that $n\geq 4\theorparam$ and $\theorparam=Cn \ln^2d/\ln n$, we obtain
$$
  |\enn\vert \leq \vert \enn\cap \e_{\ref{th: deloc}}\vert + \vert \e_{\ref{th: deloc}}^c\vert
 \leq \frac{12\theorparam}{n} \, \vert \en\vert +  2\vert \e_{\ref{th: deloc}}^c\vert
 \leq \frac{12 C  \ln^2d}{\ln n} \, \vert \en\vert +  2\vert \e_{\ref{th: deloc}}^c\vert .
$$
Theorem~\ref{th: deloc} implies the desired  result.
\end{proof}

\begin{rem}
For  $\A\in E_{n-1}\cap \e_{\ref{th: deloc}}$, Lemma~\ref{n-2images} guarantees existence of
many simple switchings which produce full rank matrices from $\A$.
With the above notations, we have
$$
 \vert Q_{n-1}(\A)\vert\geq n(n-3\theorparam)d^2/2.
$$
In order to prove along the same lines that a ``typical'' matrix in $\Aset$ is non-singular,
one needs to consider the reverse operation as well, i.e., to show that for any
full rank matrix, there are very few switchings which transform it to a singular one.
The argument of this note is based on finding switchings using structural information about vectors in the kernel,
specifically, delocalization properties in Theorem~\ref{th: deloc}.
When the matrix is of full rank, we do not have any non-trivial null vectors at hand,
which does not allow to revert the above procedure and verify invertibility.
\end{rem}

\subsection*{Acknowledgments}
P.Y. was supported by grant ANR-16-CE40-0024-01.
A significant part of this work was completed while the last three named authors were in residence at the
Mathematical Sciences Research Institute in Berkeley, California,
supported by NSF grant DMS-1440140, and the first two named authors visited the institute. The hospitality of MSRI and of the
organizers of the program on Geometric Functional Analysis and
Applications is gratefully acknowledged.

\address

\end{document}